\documentclass[a4, 11pt] {amsart} 
\usepackage{mathrsfs}
\usepackage{amsmath}
\usepackage{amssymb}
\usepackage[all]{xy}
\usepackage[dvips]{graphicx, color}
\usepackage{amsmath}
\usepackage{amssymb}
\usepackage[all]{xy}

\usepackage{graphicx}

\begin{document}

\numberwithin{equation}{section}
\newtheorem{theorem}[equation]{Theorem}
\newtheorem{proposition}[equation]{Proposition}
\newtheorem{problem}[equation]{Problem}
\newtheorem{qu}[equation]{Question}
\newtheorem{corollary}[equation]{Corollary}
\newtheorem{con}[equation]{Conjecture}
\newtheorem{lemma}[equation]{Lemma}
\theoremstyle{definition}
\newtheorem{ex}[equation]{Example}
\newtheorem{defn}[equation]{Definition}
\newtheorem{ob}[equation]{Observation}
\newtheorem{remark}[equation]{Remark}

\hyphenation{homeo-morphism} 

\newcommand{\calA}{\mathcal{A}} 
\newcommand{\calD}{\mathcal{D}} 
\newcommand{\calE}{\mathcal{E}}
\newcommand{\calC}{\mathcal{C}} 
\newcommand{\Set}{\mathcal{S}et\,} 
\newcommand{\Top}{\mathcal{T}\!op \,}
\newcommand{\Topst}{\mathcal{T}\!op\, ^*}
\newcommand{\calK}{\mathcal{K}} 
\newcommand{\calO}{\mathcal{O}} 
\newcommand{\calS}{\mathcal{S}} 
\newcommand{\calT}{\mathcal{T}} 
\newcommand{\Z}{{\mathbb Z}}
\newcommand{\C}{{\mathbb C}}
\newcommand{\Q}{{\mathbb Q}}
\newcommand{\R}{{\mathbb R}}
\newcommand{\N}{{\mathbb N}}
\newcommand{\F}{{\mathcal F}} 
\def\op{\operatorname}

\hfill

\title{A remark on periods  
of periodic sequences modulo $m$
}

\author{Shoji Yokura}

\date{}

\address{Graduate School of Science and Engineering, Kagoshima University, 1-21-35 Korimoto, Kagoshima, 890-0065, Japan
}

\email{yokura@sci.kagoshima-u.ac.jp}

\maketitle

\begin{abstract} Let $\{G_n\}$ be a periodic sequence of integers modulo $m$ and let $\{SG_n\}$ be the partial sum sequence defined by $SG_n:= \sum_{k=0}^nG_k $ (mod $m$).  
We give a formula for the period of $\{SG_n\}$.
We also show that for a generalized Fibonacci sequence $F(a,b)_n$ such that $F(a,b)_0=a$ and $F(a,b)_1=b$, we have 
$$S^i F(a,b)_n  = S^{i-1}F(a,b)_{n+2}-{n+i \choose i-2}a-{n+i \choose i-1}  b$$
where $S^i F(a,b)_n $ is the i-th partial sum sequence successively defined by $S^i F(a,b)_n := \sum_{k=0}^n S^{i-1}F(a,b)_k$. This is a generalized version of the well-known formula 
$$\sum_{k=0}^n F_k = F_{n+2} -1$$
 of the Fibonacci sequence $F_n$.

\end{abstract}


\section{Introduction}
We denote a sequence $\{a_n\}$ simply by $a_n$, without brackets, unless some confusion is possible.
Given a sequence $a_n$, its partial sum sequence $Sa_n$ is defined by 
$$Sa_n:=\sum_{k=0}^n a_k.$$

The Fibonacci sequence $\{F_n \}$ is defined by 
$F_0=0, F_1=1, F_n= F_{n-1}+F_{n-2}\,  (n \geq 2)\footnote{Sometimes the Fibonacci sequence is defined to start at $n=1$ instead of at $n=0$, namely, defined by $F_1=F_2=1$ and $F_n= F_{n-1}+F_{n-2}\,  (n \geq 3)$, e.g., see \cite[Example 4.27]{BG}.}:$
$$0, 1, 1, 2, 3, 5, 8, 13, 21, 34, 55, 89, 144, 233, 377, 610, 987, 1597, 2584, 4181, 6765, \ldots.$$

 It is well-known (e.g., see \cite{FM, Re, Sh, Vi, Wa}) that, for any integer $m \geq 2$, the Fibonacci sequence $F_n$ modulo $m$ is a periodic sequence. For example, for $m=2$ and $m=6$, the Fibonacci sequences $F_n$ modulo $2$ and $6$,  and the partial sum sequences $SF_n$ modulo $2$ and $6$ are the following:
 \begin{itemize}
\item $m=2$: 

$\, \, \, \, F_n : \underbrace{011}\underbrace{011}\underbrace{011}\underbrace{011}, \cdots \cdots$ 

$SF_n : \underbrace{010}\underbrace{010}\underbrace{010}\underbrace{010}, \cdots \cdots$ 

\item $m=6$: 

$\, \, \, \,  F_n: \underbrace{011235213415055431453251}\underbrace{011235213415055431453251}  \cdots \cdots$ 

$SF_n: \underbrace{012410230454432034214050}\underbrace{012410230454432034214050} \cdots \cdots$ 
\end{itemize}
The periods of $F_n$ modulo $2$ and $SF_n$ modulo $2$ are the same number, $3$,  and also the periods of $F_n$ modulo $6$ and $SF_n$ modulo $6$ are the same number, $24$.  In fact,  this holds for any integer $m \geq 2$, as we will see.
Furthermore, if we consider $S^2F_n:=S(SF_n)$ (which will be discussed in Section 4), then the period of $S^2F_n$ is not necessarily the same as that of $F_n$.

Now, let us consider the following same periodic sequence $G_n$ modulo $m=2, 3,4,5,6,7,8$ (its period is 3):
$$G_n:\underbrace{011}\underbrace{011}\underbrace{011} \cdots \cdots$$
Then its partial sum sequence $SG_n$ becomes as follows:
\begin{enumerate}
\item $m=2:SG_n : \underbrace{010}\underbrace{010}\underbrace{010}\underbrace{010}\cdots$; period is $3=3 \times 1$
\item $m=3:SG_n : \underbrace{012201120}\underbrace{012201120}\underbrace{012201120}\underbrace{012201120} \cdots$; period is $9=3 \times 3$.
\item $m=4:SG_n : \underbrace{012230} \underbrace{012230} \underbrace{012230} \underbrace{012230}\cdots$; period is $6 = 3\times 2$.

\item $m=5:SG_n : \underbrace{012234401123340}\underbrace{012234401123340} \cdots$; period is $15=3 \times 5$.

\item $m=6:SG_n : \underbrace{012234450}\underbrace{012234450} \underbrace{012234450} \cdots$; period is $9=3 \times 3$.

\item $m=7:SG_n : \underbrace{012234456601123345560}\cdots$; period is $21=3 \times 7$.

\item $m=8:SG_n : \underbrace{012234456670}\underbrace{012234456670}\cdots$; period is $12=3 \times 4$.
\end{enumerate}
In this note, given any periodic sequence $G_n$ modulo $m$ (not necessarily the Fibonacci sequence modulo $m$), we give a formula of the period of the partial sum sequence $SG_n$ in terms of the period of the original sequence $G_n$.
\begin{theorem} Let $p$ be the period of a periodic sequence $G_n$ modulo $m$. Then the period of the partial sum sequence $SG_n$ modulo $m$ is equal to
$$p \times \op{ord} \bigl (SG_{p-1} \op{mod} m \bigr),$$
where $\op{ord}(x)$ denotes the order of an element $x$ of $\mathbb Z_m$.
In particular, if $m$ is a prime, then the period of the partial sum sequence $SG_n$ modulo $m$ is 
equal to 
$$
\begin{cases}
\quad p \quad \quad \text{if $SG_{p-1} \equiv 0 \,\, \op{mod} \, \, m,$} \\ 
p  \times m \quad \text{if $SG_{p-1} \not \equiv 0 \,\, \op{mod} \, \, m.$}
\end{cases}
$$
\end{theorem}

For example, in the above examples, $SG_{2} \op {mod} m$ is 
equal to $2$ for each $m$, and the order $\op{ord} \bigl (SG_{p-1} \op{mod} m \bigr)$ is as follows:
\begin{enumerate}
\item $m=2$; $\op{ord} \bigl (SG_{2} \op{mod} 2 \bigr)=1$.
\item $m=3$; $\op{ord} \bigl (SG_{2} \op{mod} 3 \bigr)=3$.
\item $m=4$; $\op{ord} \bigl (SG_{2} \op{mod} 4 \bigr)=2$.

\item $m=5$; $\op{ord} \bigl (SG_{2} \op{mod} 5 \bigr)=5$.

\item $m=6$; $\op{ord} \bigl (SG_{2} \op{mod} 6 \bigr)=3$.

\item $m=7$; $\op{ord} \bigl (SG_{2} \op{mod} 7 \bigr)=7$.

\item $m=8$; $\op{ord} \bigl (SG_{2} \op{mod} 8 \bigr)=4$.
\end{enumerate}
\section{General Fibonacci sequence $F(a,b)_n$ modulo $m$}
In this section we observe that the Fibonacci sequence $F_n$ modulo $m$ and its partial sum sequence $SF_n$ modulo $m$ have the same period.

\begin{defn} A sequence $\widetilde F_n$ satisfying 
$$\widetilde F_n =\widetilde F_{n-1} +\widetilde F_{n-2}\quad  (n \geq 2) 
$$
is called \emph{a general Fibonacci sequence} .
\end{defn}

For such a general Fibonacci sequence $\widetilde F_n$, $\widetilde F_0 =a$ and $\widetilde F_1=b$ can be any integers. Therefore, we denote such a general Fibonacci sequence $\widetilde F_n$ by $F(a,b)_n$. In \cite{Kn} it is denoted by $G(a,b,n)$, and the sequence $F(2,1)_n$ is the \emph{Lucas sequence} and is denoted by $L_n$.

For a positive integer $m \geq 2$, a general Fibonacci sequence modulo $m$ is defined as follows.
\begin{defn} Let $a, b \in \mathbb Z_m =\{ 0, 1, 2, \cdots m-1\}$. The sequence $\widetilde F_n$ is defined by
\begin{itemize}
\item $\widetilde F_0=a, \widetilde F_1=b$:
\item$\widetilde F_n \equiv \widetilde F_{n-1} +\widetilde F_{n-2}$ mod $m$ for all $n \geq 2$.
\end{itemize}
\end{defn}
\noindent
Such a sequence modulo $m$ may be denoted by $F(a,b)_n \, (\op{mod} \,  m)$, but to avoid messy notation, we use the same symbol $F(a,b)_n$.
The following periodicity of a general Fibonacci sequence modulo $m$ is well-known (e.g., see \cite{Re, Sh, Vi, Wa}.)
\begin{theorem} For any integer $m \geq 2$, a general Fibonacci sequence $F(a,b)_n$ modulo $m$ is a periodic sequence.
\end{theorem}
The period of a general Fibonacci sequence $F(a,b)_n$ modulo $m$ shall be denoted by \newline $\Pi(F(a,b)_n, m)$. For the Fibonacci sequence $F_n$ modulo $m$, the period $\Pi(F_{n},m)$ is called the \emph{Pisano\footnote{Pisano is another name for Fibonacci (cf. Wikipedia \cite{FW}).}  period} of $m$ and is denoted by $\pi(m)$.
For some interesting results of the Pisano period $\pi(m)$, see, e.g., \cite{Re}. For example, one can see a list of the periods $\pi(m)$ for $2 \leq m \leq 2001$; for instance,
$$\pi(10)=60, \pi(25)=100, \pi(98)=336, \pi(250)=1500, \pi(500)=1500, $$
$$\pi(625)=2500, \pi(750)=3000, \pi(987)=32, \pi(1250)=7500,$$
$$  \pi(1991)=90, \pi(2000)=3000, \pi(2001)=336.$$

As to relations between $\pi(m)$ and $m$, 
the following are known.
\begin{theorem}[Freyd and Brown \cite{Br}] 
\begin{enumerate}
\item $\pi(m) \leq 6m$ with equality if and only if $m =2 \times 5^k (k=1, 2, 3, \cdots).$
\item For the Lucas number $L_{n}$, $\Pi(L_n, m) \leq 4m$ with equality if and only if $m=6$.
\end{enumerate}
\end{theorem}

Given a sequence $a_n$, the partial sum sequence $b_n$ defined by $b_n:= \sum_{k=0}^n a_k$ shall be called \emph{the first derived sequence}, instead of ``partial sum sequence''. 
The first derived sequence of a general Fibonacci sequence $F(a,b)_n$ shall be denoted by $SF(a,b)_n:= \sum_{k=0}^{n} F(a,b)_k$.

\begin{lemma} \label{SF-F} For a general Fibonacci sequence $F(a,b)_n$ the following holds:
\begin{equation}\label{SF}
SF(a,b)_n = F(a,b)_{n+2} - b.
\end{equation} 
\end{lemma}
\begin{proof} 
We have $SF(a,b)_0=F(a,b)_0 = F(a,b)_2 - F(a,b)_1 =F(a,b)_2 - b.$
So, we suppose that
$SF(a,b)_n=F(a,b)_{n+2} - b.$
Then 
\begin{align*}
SF(a,b)_{n+1} & =SF(a,b)_n + F(a,b)_{n+1}\\
& = F(a,b)_{n+2}-b + F(a,b)_{n+1}\\
& = F(a,b)_{n+2} +F(a,b)_{n+1} - b\\
&= F(a,b)_{n+3} - b\\
& = F(a,b)_{(n+1)+2} -b.
\end{align*}
\end{proof}

Equation (\ref{SF}) 
is a generalized version of 
the following well-known formula of  the Fibonacci sequence \cite{Lu}:
$$ \sum_{k=0}^n F_k = F_{n+2} -1.$$
From Equation (\ref{SF}), we get the following 
\begin{corollary}\label{cor} For any positive integer $m \geq 2$, a general Fibonacci sequence $F(a,b)_n$ modulo $m$ and its first derived Fibonacci sequence $SF(a,b)_n$ modulo $m$ have the same period:
$$\Pi (F(a,b)_n, m ) = \Pi (SF(a,b)_n, m ).$$
\end{corollary}

\begin{lemma}\label{lemma} If $\Pi (F(a,b)_n, m ) =p$, then we have
\begin{equation}\label{mod m}
SF(a,b)_{p-1} = \sum_{k=0}^{p-1} F(a,b)_k \equiv 0  \, \, \op{mod} m.
\end{equation}
\end{lemma}
\begin{proof} There are two ways to prove this. 

(I) It follows from Lemma \ref{SF-F} that
$SF(a,b)_{p-1} = F(a,b)_{p+1}-b.$
Since the period of the general Fibonacci sequence $F(a,b)_n$ modulo $m$ is $p$, we have 
$F(a,b)_{p+1} \equiv F(a,b)_1.$
Hence,
\begin{align*}
SF(a,b)_{p-1} & =F(a,b)_{p+1}-b\\
& \equiv F(a,b)_1 -b = 0 \, \, \op{mod} m.
\end{align*}

(II) By Corollary \ref{cor} $\Pi (SF(a,b)_n, m )=p$. Hence,
$SF(a,b)_0 \equiv SF(a,b)_p$, namely,
$$F(a,b)_0 \equiv \sum_{k=0}^pF(a,b)_k = \sum_{k=0}^{p-1}F(a,b)_k + F(a,b)_p.$$
Since $F(a,b)_0 \equiv F(a,b)_p \,\, \op{mod} \, \, m$, we get
$$ SF_{p-1}= \sum_{k=0}^{p-1}F(a,b)_k \equiv 0 \, \, \op{mod} \, \, m.$$
\end{proof}
Indeed, in the following examples we do see the above mod formula (\ref{mod m}).
\begin{ex} \label{ex1} For the Fibonacci sequence $F_n$ modulo $m$
\begin{enumerate}
\item $m=2$: $\pi(2) =3$.
$\underbrace{011}\underbrace{011}\underbrace{011}\underbrace{011}\cdots \cdots$ 
\item $m=3$: $\pi(3) =8$.
$\underbrace{01120221}\underbrace{01120221}\underbrace{01120221}\underbrace{01120221} \cdots \cdots$ 
\item $m=4$: $\pi(4) =6$.
\, $\underbrace{011231}\underbrace{011231}\underbrace{011231}\underbrace{011231}\cdots \cdots$ 
\item $m=5$: $\pi(5) =20$.
$\underbrace{01123033140443202241}\underbrace{01123033140443202241} \cdots \cdots$ 
\item $m=6$: $\pi(6) =24$.
$\underbrace{011235213415055431453251}\underbrace{011235213415055431453251} \cdots$ 
\item $m=7$: $\pi(7) =16$.
$\underbrace{0112351606654261}\underbrace{0112351606654261}\underbrace{0112351606654261}\cdots$ 
\item $m=8$: $\pi(8) =12$.
$\underbrace{011235055271}\underbrace{011235055271}\underbrace{011235055271}\underbrace{011235055271}\cdots$ 
\end{enumerate}
\end{ex}

\section{Periods of periodic sequences modulo $m$}
The first derived sequence $SF(a,b)_n$ modulo $m$ of a general Fibonacci sequence $F(a,b)$ modulo $m$ is periodic and has the same period as that of the general Fibonacci sequence $F(a,b)$ modulo $m$. As seen in the introduction, it is not the case for an arbitrary periodic sequence modulo $m$. So, in this section we consider periods of the first derived sequence of a periodic sequence modulo $m$.

\begin{proposition} \label{prop} Let $G_n$ be a periodic sequence 
modulo $m$ with period $p$. If 
$SG_{p-1}  \equiv 0 \, \, \op{mod} m,$
then for any $n$ we have
$SG_n \equiv SG_{n+p} \op{mod} m.$
\end{proposition}
\begin{proof}
The period of $G_n$ is $p$, thus $G_0 \equiv G_p $. Since $SG_0 = G_0$, we have
$SG_0 \equiv G_p \, \, \op{mod} m.$
Since $SG_{p-1}  =\sum_{k=0}^{p-1} G_{k} \equiv 0 \, \, \op{mod} m$, 
$$SG_0 \equiv SG_{p-1}  + G_p  = SG_{p} \, \, \op{mod} m.$$
So, suppose that for $n$ we have
$SG_n \equiv SG_{n+p}$.
Then 
\begin{align*}
SG_{n+1} & \equiv SG_n + G_{n+1} \, \, \op{mod} m\\
& \equiv SG_{n+p} + G_{n+1+p} \, \, \op{mod} m \, \, \text{(since $G_{n+1} \equiv G_{n+1+p} \, \, \op{mod} m$)}\\
& \equiv SG_{n+1+p} \, \, \op{mod} m.
\end{align*}
Hence, by induction,  we get the statement of the proposition.
\end{proof}
Here it should be noticed  that from Proposition \ref{prop} one \emph{cannot} automatically claim that the period of the first derived sequence $SG_n$ modulo $m$ is $p$; one can claim only that the period of $SG_n$ is \emph{a divisor of the period of the original sequence $G_n$}:
$$\Pi(SG_n, m) \, \Bigl | \, \Pi(G_n, m).$$
In fact, we can show 
\begin{theorem}\label{key theorem} Let $G_n$ be a periodic sequence 
modulo $m$ with period $p$. If 
$SG_{p-1}  \equiv 0 \, \, \op{mod} \, \, m$, then $\Pi(SG_n, m) =\Pi(G_n, m).$
\end{theorem}
\begin{proof} 
Suppose that $\Pi (SG_n, m ) \not = \Pi (G_n, m  )$, namely $\Pi (SG_n, m )$ is a proper divisor of $p= \Pi (G_n, m )$. Let $d =\Pi (SG_n, m ) < p$ and $d \, | \, p.$ Thus, we have that for all $n \geq 0$
\begin{equation}\label{d}
SG_n \equiv  SG_{n+d} \, \, \op{mod} \, \,  m.
\end{equation}
We have that $SG_0 \equiv  SG_d\, \, \op{mod} \, \,  m$  implies $G_1 + \cdots + G_d \equiv 0 \, \, \op{mod} \, \, m$, thus
$$G_1 \equiv -(G_2 + \cdots + G_d ) \, \, \op{mod} \, \, m.$$
Now, $SG_1 \equiv  SG_{1+d}\, \, \op{mod} \, \,  m$  implies $G_2  + \cdots +G_d + G_{1+d} \equiv 0 \, \, \op{mod} \, \, m$, thus
$$G_{1+d} \equiv -(G_2 + \cdots + G_d ) \, \, \op{mod} \, \, m.$$
Therefore we get 
$$G_1 \equiv G_{1+d}  \, \, \op{mod} \, \, m.$$
Let $p= d\times p_0$. Continuing this procedure, we get the following congruences:
\begin{itemize}
\item $G_1 \equiv G_{1+d} \equiv G_{1+2d} \cdots \equiv G_{1+(p_0-1)d} \, \, \op{mod} \, \, m.$
\item $G_2 \equiv G_{2+d} \equiv G_{2+2d} \cdots \equiv G_{2+(p_0-1)d} \, \, \op{mod} \, \, m.$
\item $G_3 \equiv G_{3+d} \equiv G_{3+2d} \cdots \equiv G_{3+(p_0-1)d} \, \, \op{mod} \, \, m.$

       $\cdots \cdots \cdots \cdots \cdots \cdots \cdots \cdots \cdots \cdots \cdots $
       
              $\cdots \cdots \cdots \cdots \cdots \cdots \cdots \cdots \cdots \cdots \cdots $
       
\item $G_{d-1} \equiv G_{2d-1} \equiv G_{3d-1} \cdots \equiv G_{p_0d -1} \, \, \op{mod} \, \, m.$
\item $G_d \equiv G_{2d} \equiv G_{3d} \equiv \cdots \equiv G_{p_0d} =G_p \, \, \op{mod} \, \, m.$
\end{itemize}
Since $G_0\equiv G_p \,\, \op{mod} \, \, m$, the final congruences $G_d \equiv G_{2d} \equiv G_{3d} \equiv \cdots \equiv G_{p_0d} =G_p \, \, \op{mod} \, \, m$ 
become
\begin{itemize}
\item $G_0 \equiv G_d \equiv G_{2d} \equiv G_{3d} \equiv \cdots \equiv G_p \, \, \op{mod} \, \, m.$
\end{itemize}
Hence $\Pi(G_n, m) =d <p$, which contradicts the fact that $\Pi(G_n, m) =p$.
\end{proof}
As in the proof (II) of Lemma \ref{lemma}, if $\Pi(G_n, m)=\Pi(SG_n,m)=p$, then we have $SG_{p-1}  \equiv 0 \, \, \op{mod} \, \, m$. Therefore we get the following
\begin{corollary}Let $G_n$ be a periodic sequence 
modulo $m$ with period $p$. Then we have 
$SG_{p-1}  \equiv 0 \, \, \op{mod} \, \, m$ if and only if $\Pi(SG_n, m) =\Pi(G_n, m).$
\end{corollary}
Next,  we consider the case when $SG_{p-1}  \not \equiv 0 \, \, \op{mod} m $.
\begin{theorem} Let $G_n$ be a periodic sequence 
modulo $m$ with period $p$. 
If $SG_{p-1}  \not \equiv 0 \, \, \op{mod} \, \, m$, then 
$$\Pi(SG_n,m) = s \times \Pi(G_n,m),$$
where $s$ is the order of $SG_{p-1} (\op{mod} \, m)$ in $\mathbb Z_m$, i.e., $s$ is the smallest non-zero integer such that 
$$ s\times SG_{p-1} \equiv 0 \, \, \op{mod} \, \, m.$$
\end{theorem}
\begin{proof} First we observe that for all $i \geq 2$
\begin{equation*}
SG_{ip-1} \equiv  SG_{(i-1)p-1} + SG_{p-1} \, \, \op{mod} \, \, m
\end{equation*}
from which we obtain that for all $i \geq 1$
\begin{equation*}
SG_{ip-1} \equiv i \times SG_{p-1} \, \, \op{mod} \, \, m
\end{equation*}
Indeed, the first $p$-tuple is $\{SG_0, SG_1, \cdots, SG_{p-1}\}$ and the second $p$-tuple 
is \newline
$\{SG_p, SG_{p+1}, \cdots, SG_{2p-1} \}$, which is, modulo $m$,  the same as 
$$\{SG_{p-1}+ SG_0, SG_{p-1}+SG_1, \cdots, SG_{p-1}+SG_{p-1}\}.$$
Hence, $SG_{2p-1}\equiv SG_{p-1}+SG_{p-1} \, \, \op{mod} \, \, m$, thus we get 
\begin{equation}\label{eq-2}
SG_{2p-1}\equiv 2 \times SG_{p-1} \, \, \op{mod} \, \, m.
\end{equation}
 The third $p$-tuple
$\{SG_{2p}, SG_{2p+1}, \cdots, SG_{3p-1}\}$ is, modulo $m$, the same as 
$$\{SG_{2p-1}+ SG_0, SG_{2p-1}+SG_1, \cdots, SG_{2p-1}+SG_{p-1}\}.$$
Hence, $SG_{3p-1} \equiv SG_{2p-1} + SG_{p-1}$, thus from Equation (\ref{eq-2}) we get
\begin{equation*}\label{eq-3}
SG_{3p-1}\equiv 3 \times SG_{p-1} \, \, \op{mod} \, \, m.
\end{equation*}
Now, let us suppose that 
\begin{equation}\label{eq-j}
SG_{jp-1}\equiv j \times SG_{p-1} \, \, \op{mod} \, \, m.
\end{equation}
We see that the $(j+1)$-th $p$-tuple 
$\{SG_{jp}, SG_{jp+1}, \cdots, SG_{(j+1)p-1}\}$ is, modulo $m$, the same as 
$$\{SG_{jp-1}+ SG_0, SG_{jp-1}+SG_1, \cdots, SG_{jp-1}+SG_{p-1}\}.$$
Hence, $SG_{(j+1)p-1} \equiv SG_{jp-1} + SG_{p-1} \, \, \op{mod} \, \, m$, and thus from Equation (\ref{eq-j}) we get
$$SG_{(j+1)p-1}\equiv (j+1) \times SG_{p-1} \, \, \op{mod} \, \, m..$$
Hence, by induction, we have that for all $i \geq 1$
$$SG_{ip-1} \equiv i \times SG_{p-1} \, \, \op{mod} \, \, m.$$
Since $ s\times SG_{p-1} \equiv 0 \, \, \op{mod} \, \, m$, we have
$$SG_{ps -1} \equiv \sum_{k=0}^{ps-1} G_k \equiv 0 \, \, \op{mod} \, \, m.$$
As in the proof of Proposition \ref{prop}, we see that for any $n$
$$SG_n \equiv SG_{n+ps} \, \, \op{mod} \, \, m.$$
Hence the period of $SG_n$ is a divisor of $ps$. Suppose that such a divisor is a proper one, denoted by $\delta$. Then, as in the proof of Theorem \ref{key theorem}, we have 
$$G_n \equiv G_{n+\delta} \, \, \op{mod} \, \, m.$$
Since the period of $G_n$ is $p$, $\delta$ has to be a multiple of $p$, thus $\delta = \omega p$ for some non-zero integer $\omega$. Since $\delta =\omega p$ is a proper divisor of $ps$, $\omega$ is a proper divisor of $s$, in particular $\omega < s$. Then,  as in the proof (II) of Lemma \ref{lemma},  $SG_n \equiv GS_{n+\omega p} \, \, \op{mod} \, \, m$ implies that
$$GS_{\omega p-1}= \sum _{k=0}^{\omega p-1}G_k \equiv 0 \, \, \op{mod} \, \, m.$$
In other words
$$\omega \times SG_{p-1} \equiv 0 \, \, \op{mod} \, \, m.$$
This contradicts the fact that $s$ is the smallest non-zero integer such that $s\times SG_{p-1} \equiv 0 \, \, \op{mod} \, \, m.$ Hence the period $\delta$ of $SG_n$ has to be exactly $ps$, i.e., $$\Pi(SG_n,m) = s \times \Pi(G_n,m).$$

\end{proof}
Therefore we get the following theorem:
\begin{theorem}\label{main theorem} Let $p$ be the period of a periodic sequence $G_n$ modulo $m$. Then the period of the partial sum sequence $SG_n$ modulo $m$ is equal to
$$\op{ord} \bigl (SG_{p-1} \op{mod} m \bigr) \times p,$$
where $\op{ord}(x)$ denotes the order of an element $x$ of $\mathbb Z_m$.

In particular, if $m$ is a prime, then the period of the partial sum sequence $SG_n$ modulo $m$ is 
equal to 
$$
\begin{cases}
\quad p \quad \quad \text{if $SG_{p-1} \equiv 0 \,\, \op{mod} \, \, m,$} \\ m \times p  \quad \text{if $SG_{p-1} \not \equiv 0 \,\, \op{mod} \, \, m.$}
\end{cases}
$$
\end{theorem}

We note that for an element $n \in \mathbb Z_m$, the order $\op{ord}(n)$ of the element $n$ is given by
$$\op{ord}(n)= \frac{\op{LCM}(n,m)}{n}$$
where $\op{LCM}(n,m)$ is the least common multiple of $n$ and $m$.

\begin{ex} Let us consider the following periodic sequence $G_n$ in $\mathbb Z_m$ with $m=15, 30, 36$
$$G_n: \underbrace{20190823}\underbrace{20190823}\underbrace{20190823} \cdots \cdots.$$
Then $\Pi(G_n,m)=8$ for any $m$ and we have that $2+0+1+9+0+8+2+3=25$.
\begin{enumerate}
\item $\Pi(SG_n, 15) = 8 \times \frac{\op{LCM}(25,15)}{25}=8 \times \frac{75}{25}=8\times 3=24.$
\item $\Pi(SG_n, 30) = 8 \times \frac{\op{LCM}(25,30)}{25}=8 \times \frac{150}{25}=8\times 6=48.$
\item $\Pi(SG_n, 36) = 8 \times \frac{\op{LCM}(25,36)}{25}=8 \times \frac{25\times 36}{25}=8\times 36=288.$
\end{enumerate}

\end{ex}
\section{Higher derived general Fibonacci sequences $S^iF(a,b)_n$ }
Given a sequence $A_n$, for a non-negative integer $i$ we define the $i$-th derived sequence inductively as follows:
$$S^iA_n := \sum_{k=0}^n S^{i-1}A_k,  \quad i \geq 1$$
where $S^0A_n := A_n.$

Thus for a general Fibonacci sequence $F(a,b)_n$,  we can consider the $i$-th derived sequence $S^iF(a,b)_n$. 
By tedious computation we can show the following formulas, which are further extensions of Lemma \ref{SF-F}.
\begin{proposition}\label{form-SF-F}  For all integers $a, b$ and $n \geq 0$, the following formulas hold.
\begin{enumerate}
\item $S^2F(a,b)_n = SF(a,b)_{n+2}-a-(n+2)b.$
\item $ \displaystyle S^3F(a,b)_n = S^2F(a,b)_{n+2}-(n+3) a-\frac{(n+2)(n+3)}{2} b.$
\item $ \displaystyle S^4F(a,b)_n = S^3F(a,b)_{n+2}-\frac{(n+3)(n+4)}{2}a-\frac{(n+2)(n+3)(n+4)}{2\times 3} b.$
\end{enumerate}
\end{proposition}
These formulas are re-expressed as follows:
\begin{align*}
S^2F(a,b)_n  & = SF(a,b)_{n+2}-{n+2 \choose 0} a-{n+2 \choose 1} b\\
& = SF(a,b)_{n+2}-{n+2 \choose 2-2} a-{n+2 \choose 2-1} b.
\end{align*}
\begin{align*}
S^3F(a,b)_n & = S^2F(a,b)_{n+2}-{n+3 \choose 1}  a-{n+3 \choose 2}  b\\
& = S^2F(a,b)_{n+2}-{n+3 \choose 3-2}  a-{n+3 \choose 3-1}  b.
\end{align*}
\begin{align*}
S^4F(a,b)_n & = S^3F(a,b)_{n+2}-{n+4 \choose 2}a-{n+4 \choose 3}  b\\
& = S^3F(a,b)_{n+2}-{n+4 \choose 4-2}  a-{n+4 \choose 4-1}  b.
\end{align*}
So, by these re-expressions, it is natural to think that the following general formula would hold and it turns out that it is the case, i.e., we have
\begin{theorem} For $i \geq 1$ we have 
\begin{equation}\label{gene-form}
S^i F(a,b)_n  = S^{i-1}F(a,b)_{n+2}-{n+i \choose i-2}a-{n+i \choose i-1}  b.
\end{equation}
When $i=1$ we set ${n+1 \choose 1-2} = {n+1 \choose -1}=0$.
\end{theorem}
\begin{proof} The proof is of course 
by induction. Since the formula is already proved in the cases when $i=1,2,3,4$ as above, we suppose that the above formula (\ref{gene-form}) holds for $i=j$:
\begin{equation}\label{gene-form-j}
S^j F(a,b)_n  = S^{j-1}F(a,b)_{n+2}-{n+j \choose j-2}a-{n+j \choose j-1}  b
\end{equation}
 and we show the formula for $i=j+1$:
\begin{equation}\label{gene-form-j+1}
S^{j+1} F(a,b)_n  = S^jF(a,b)_{n+2}-{n+j+1 \choose j-1}a-{n+j+1 \choose j}  b.
\end{equation}
First we have
\begin{align*}
& S^{j+1} F(a,b)_n  \\
& = \sum_{k=0}^n S^jF(a,b)_k \quad \text{(by the definition of $S^{j+1} F(a,b)_n$ )} \\
& = \sum_{k=0}^n \left \{ S^{j-1}F(a,b)_{k+2}-{k+j \choose j-2}a-{k+j \choose j-1}  b \right \} \quad \text{(by Equation (\ref{gene-form-j}))}\\
& = \sum_{k=0}^n S^{j-1}F(a,b)_{k+2}-\sum_{k=0}^n{k+j \choose j-2}a- \sum_{k=0}^n{k+j \choose j-1}  b \\
& = \left (\sum_{k=0}^{n+2} S^{j-1}F(a,b)_k - S^{j-1}F(a,b)_0 - S^{j-1}F(a,b)_1 \right) 
-\sum_{k=0}^n {k+j \choose j-2}a- \sum_{k=0}^n {k+j \choose j-1}b \\
& = S^jF(a,b)_{n+2} - S^{j-1}F(a,b)_0 - S^{j-1}F(a,b)_1 -\sum_{k=0}^n {k+j \choose j-2}a- \sum_{k=0}^n {k+j \choose j-1}b.
\end{align*}
Here we observe that for all $j \geq 1$ we have $S^{j-1}F(a,b)_0=a$, which is obvious, and 
\begin{equation}\label{SF-F-1}
S^{j-1}F(a,b)_1 = (j-1)a+b.
\end{equation}
Equation (\ref{SF-F-1}) can be seen by induction as follows. If $j=1$, then $S^{j-1}F(a,b)_1 =S^0F(a,b)_1=F(a,b)_1 =b$.
So, suppose that $S^{j-1}F(a,b)_1 = (j-1)a+b$ holds and we show that $S^jF(a,b)_1 = ja+b$. Indeed,
\begin{align*}
S^jF(a.b)_1 & = S^{j-1}F(a,b)_0+ S^{j-1}F(a,b)_1 \quad \text{(by the definition of $S^jF(a.b)_1$)}\\
& = a + (j-1)a+b \quad \text {(by the above)}\\
&= ja+b.
\end{align*}
Hence we have
\begin{align*}
& S^{j+1} F(a,b)_n  \\
& = S^jF(a,b)_{n+2} - a - \{(j-1)a +b\} -\sum_{k=0}^n {k+j \choose j-2}a- \sum_{k=0}^n {k+j \choose j-1}b\\
& = S^jF(a,b)_{n+2} - \left \{ j + \sum_{k=0}^n {k+j \choose j-2} \right \} a - \left \{1+ \sum_{k=0}^n {k+j \choose j-1}\right \} b.
\end{align*}
We want to show
\begin{equation}\label{equ-1}
j + \sum_{k=0}^n {k+j \choose j-2} = {n+j+1 \choose j-1},
\end{equation}
\begin{equation}\label{equ-2}
1+ \sum_{k=0}^n {k+j \choose j-1}= {n+j+1 \choose j}.
\end{equation}
To show these, we recall the following formula (Pascal's Rule):
$${n-1 \choose k-1} + {n-1 \choose k} = {n \choose k}.$$
Then we have
\begin{align*}
& j + \sum_{k=0}^n {k+j \choose j-2} \\
& = {j \choose 1} + \sum_{k=0}^n {j+k \choose 2+k} \\
& = \underbrace{{j \choose 1} + {j \choose 2}} + {j+1 \choose 3} + {j+2 \choose 4}  + \cdots + {j+n \choose 2+n}  \\
& ={j+1 \choose 2} + {j+1 \choose 3} + {j+2 \choose 4}  + \cdots + {j+n \choose 2+n} \quad \text{(using Pascal's Rule)} \\
& = \underbrace{{j+1 \choose 2} + {j+1 \choose 3}} + {j+2 \choose 4}  + \cdots + {j+n \choose 2+n}  \\
& = {j+2 \choose 3}  + {j+2 \choose 4}  + {j+3 \choose 5} + \cdots + {j+n \choose 2+n}  \\
& = \cdots \cdots \cdots \cdots \cdots \cdots \cdots \cdots \cdots \cdots \cdots    \quad \text{(using Pascal's Rule step by step )} \\
& = {j+n \choose 1+n} +{j+n \choose 2+n}\\
& = {j+n+1 \choose 2+n} \quad \text{(using Pascal's Rule)}\\
& = {j+n+1 \choose j-1}.
\end{align*}
Thus we get Equation (\ref{equ-1}).
Similarly, using Pascal's Rule step by step we get
\begin{align*}
1+ \sum_{k=0}^n {k+j \choose j-1} &  = 1+ \sum_{k=0}^n {j+k \choose 1+k} \\
& = {j \choose 0} + {j \choose 1} + {j+1 \choose 2} + {j+2 \choose 3}  + \cdots + {j+n \choose 1+n}  \\
& = \cdots \cdots \cdots \cdots \cdots \cdots \cdots \cdots \cdots \cdots \cdots    \\
& = \cdots \cdots \cdots \cdots \cdots \cdots \cdots \cdots \cdots \cdots \cdots    \\
& = {j+n \choose n} +{j+n \choose 1+n}\\
& = {j+n+1 \choose 1+n} \quad \text{(using Pascal's Rule)}\\
& = {j+n+1 \choose j}.
\end{align*}
Thus we get Equation (\ref{equ-2}).
\end{proof}

\begin{ex} \label{ex2} For each case of the above Example \ref {ex1}
\begin{enumerate}
\item In the case when $m=2$:

$\, \, \, \, \, \, \, F_n: \, \, \underbrace{011}\underbrace{011}\underbrace{011}\underbrace{011}\cdots \cdots$ 

$\, \, \, SF_n: \, \, \underbrace{010}\underbrace{010}\underbrace{010}\underbrace{010}\underbrace{010}\cdots \cdots$ 

$S^2F_n: \, \, \underbrace{011100}\underbrace{011100}\underbrace{011100}\underbrace{011100}\cdots \cdots$ 

$S^3F_n: \, \, \underbrace{010111101000}\underbrace{010111101000}\underbrace{010111101000}\underbrace{010111101000}\cdots \cdots$ 

$S^4F_n: \, \, \underbrace{011010110000}\underbrace{011010110000}\underbrace{011010110000}\underbrace{011010110000}\cdots \cdots$ 

\item In the case when $m=3$:

$\, \, \, \, \, \, \, F_n:  \, \, \underbrace{01120221}\underbrace{01120221}\underbrace{01120221}\underbrace{01120221}\cdots \cdots$ 

$\, \, \, SF_n: \, \,  \underbrace{01211020}\underbrace{01211020}\underbrace{01211020}\underbrace{01211020}\cdots \cdots$ 

$S^2F_n: \, \, \underbrace{010122111212002220201100}\underbrace{010122111212002220201100}\cdots \cdots$ 

$S^3F_n: \, \, \underbrace{011210120202221022112000}\underbrace{011210120202221022112000}\cdots \cdots$ 

$S^4F_n: \, \, {\small \underbrace{0121220221102122101211111202001002210200212022222201011211002101102010000}} \cdots $ \\
\end{enumerate}
\end{ex}
From Theorem \ref{main theorem} we get the following:

\begin{theorem} We let $p_i=\Pi \Bigl (S^iF(a,b)_n, m \Bigr)$ with $p_0= \Pi \Bigl (F(a,b)_n, m \Bigr)$.
Then we have 
$$\Pi \Bigl (S^{i+1}F(a,b)_n, m \Bigr) = \op{ord} \Bigl (S^iF(a,b)_{p_i-1} \op{mod} m \Bigr) \times \Pi \Bigl (S^iF(a,b)_n, m \Bigr).$$
In particular, if $m$ is a prime, we have
$$
\Pi \Bigl (S^{i+1}F(a,b)_n, m \Bigr) =
\begin{cases}
\quad \Pi \Bigl (S^iF(a,b)_n, m \Bigr) \quad \quad \text{if $S^iF(a,b)_{p_i-1}  \equiv 0 \,\, \op{mod} \, \, m,$} \\
m \times \Pi \Bigl (S^iF(a,b)_n, m \Bigr) \quad \text{if $S^iF(a,b)_{p_i-1}  \not \equiv 0 \,\, \op{mod} \, \, m.$}
\end{cases}
$$
\end{theorem}
\vspace{0.5cm}

\noindent
{\bf Acknowledgements}. The author would like to thank Kazuya Hamada for providing him a question which led to this present work, Marc Renault and Osamu Saeki for useful comments, the referee for his/her thorough reading and useful comments and suggestions, and Bruce Landman for his thorough editorial work. The author is supported by JSPS KAKENHI Grant Numbers JP19K03468.


\end{document}